\documentclass [12pt] {amsart}

\usepackage [latin1]{inputenc}
\usepackage[all]{xy}
\usepackage[english]{babel}
\usepackage{amssymb}
\usepackage[lite, initials]{amsrefs}

\newtheorem{teorema}{Theorem}[section]
\newtheorem*{theorem*}{Structure Theorem}
\newtheorem{lemma}[teorema]{Lemma}
\newtheorem{propos}[teorema]{Proposition}
\newtheorem{corol}[teorema]{Corollary}
\theoremstyle{definition}

\newtheorem{rem}{Remark}[section]

\def\R{{\mathbb R}}

\newcommand{\Ci}{\mathcal{C}}

\def\oli{\overline}				

\def\de{\partial}
\def\debar{\oli{\de}}

\title{Weakly complete domains in Grauert type surfaces}
\author[S. Mongodi]{Samuele Mongodi}
\address{Politecnico di Milano, Dipartimento di Matematica, Via Bonardi, 9 -- I-20133 Milano, Italy}
\email{samuele.mongodi@polimi.it}

\begin{document}
\begin{abstract}
The aim of this short note is to investigate the geometry of weakly complete subdomains of Grauert type surfaces, i.e. open connected sets $D$, sitting inside a Grauert type surface $X$, which admit a smooth plurisubharmonic exhaustion function. We prove that they are either modifications of Stein spaces or Grauert type surfaces themselves and we apply these results to the special case of Hopf surfaces.
\end{abstract}

\maketitle

\section{Introduction}

A weakly complete space is a complex space which admits a plurisubharmonic exhaustion function; if the function is strictly plurisubharmonic outside a compact set, then the space is a modification of a Stein space.

Understanding the geometry of weakly complete spaces can be thought as a generalization of the classical Levi problem; in \cite{mst}, Slodkowski, Tomassini and the author classified weakly complete surfaces which admit a real analytic plurisubharmonic exhaustion function. We refer to \cite{crass} for a review of some examples and the main steps of the proof; the result, loosely stated, says that such a complex surface is either holomorphically convex or the regular level sets of the exhaustion are Levi flat hypersurfaces foliated in dense complex curves.

These last instances are called Grauert type surfaces.

The aim of this short note is to investigate the geometry of weakly complete subdomains of Grauert type surfaces, i.e. open connected sets $D$, sitting inside a Grauert type surface $X$, which admit a smooth plurisubharmonic exhaustion function.

Instances of this question can be found for example in the works of Miebach (see \cite{mie}) and Levenberg and Yamaguchi (see \cite{LY}), where the problem is raised for Hopf surfaces.

It turns out that a weakly complete domain in a Grauert type surface can either be a modification of a Stein space or a Grauert type surface itself; in particular, this implies a kind of stability result, namely that, if the ambient space $X$ admits a real analytic plurisubharmonic exhaustion function, then every weakly complete domain $\Omega\subseteq X$ can be endowed with a \emph{real analytic} plurisubharmonic exhaustion function.

\medskip

The proof we present relies on some observations that hold true not only for Grauert type surfaces, but also for holomorphically convex ones; therefore we present our results for any weakly complete surface with a real analytic plurisubharmonic exhaustion function.

\section{The kernel of a weakly complete space}
Let $X$ be a weakly complete complex surface, endowed with a real analytic plurisubharmonic exhaustion function $\alpha:X\to\R$; suppose also that $D\subseteq X$ is a weakly complete domain, with an exhaustion function $\phi:D\to\R$ which is plurisubharmonic and smooth.

Following \cite{ST}, we denote by $\Sigma^1_X$ and $\Sigma^1_D$ their kernels; namely, $\Sigma^1_X$ is the set of points of $X$ where every plurisubharmonic exhaustion function fails to be strictly plurisubharmonic.

By \cite{ST}*{Lemma 3.1}, we can suppose that both $\alpha$ and $\phi$ are minimal for $X$ and $D$ respectively, i.e. that $\Sigma^1_X$ is exactly the set of points where $\alpha$ is not strictly plurisubharmonic and $\Sigma^1_D$ is exactly the set of points where $\phi$ is not strictly plurisubharmonic. Moreover, by \cite{mst}*{Proposition 3.1}, we can assume that both $\alpha$ and $\phi$ are non negative.

The following Lemma collects some properties of kernels and minimal functions that we will need later on.

\begin{lemma}\label{lmm_1}Let $X$, $D$, $\alpha$, $\phi$ be as above. Then
\begin{enumerate}
\item $\Sigma^1_D\subseteq\Sigma^1_X$
\item for every $p\in\Sigma^1_D$ which is a regular point for $\phi$, $d\alpha(p)=\lambda_pd\phi(p)$
\item for every $p\in\Sigma^1_D$, $d\phi(p)$ vanishes on the kernel of $\de\debar\phi(p)$.
\end{enumerate}
\end{lemma}
\begin{proof}Take a point $p\in D$; if $p\not\in\Sigma^1_X$, then the Levi form of $\alpha$ is positive definite at $p$, so $\alpha+\phi$ is a plurisubharmonic exhaustion function for $D$ which is strictly plurisubharmonic at $p$, hence $p\not\in\Sigma^1_D$. This proves (1).

Now, observe that, if $\psi$ is a non negative plurisubharmonic function and $q$ is a regular point for $\psi$, then
$$\de\debar\psi^2=2(\de\psi\wedge\debar\psi + \psi\de\debar\psi)$$
is again positive semidefinite and the kernel of the associated linear map is given by the intersection of the kernel of $\de\debar\phi$ and the kernel of $\de\psi$ (or, as $\psi$ is real valued, of $d\psi$. Therefore, if $p\in \Sigma^1_D\subseteq\Sigma^1_X$, then both $\phi$ and $\phi^2$ must both have degenerate Levi forms at $p$, hence $d\phi$ has to vanish on the kernel of $\de\debar\phi$. So (3) is proved.

Now, consider $\phi^2+\alpha^2$. In order for its Levi form to be degenerate at $p$, the kernels of $d\phi$ and $d\alpha$ and the kernels of $\de\debar\phi$ and $\de\debar\alpha$ must all coincide; hence (2) follows.
\end{proof}

From the Main Theorem in \cite{mst}, in particular it follows that, for every $p\in\Sigma^1_X$, regular point for $\alpha$, there exist an irreducible complex curve $F_p$ and an injective holomorphic map $i_p:F_p\to X$ such that $p\in C_p=i(F_p)\subseteq \Sigma^1_X\cap\{\alpha=\alpha(p)\}$. Moreover, $T_pC_p$, if defined, is the kernel of $\de\debar\alpha(p)$.

\begin{propos}\label{prp_foglia}Consider $p\in\Sigma^1_D$ and suppose that $\phi(p)$ is a regular value for $\phi$. Then $C_p\subseteq\Sigma^1_D$.\end{propos}
\begin{proof}By \cite{ST}*{Lemma 4.1}, $\Sigma^1_D\cap\{\phi=\phi(p)\}$ is foliated in complex curves; let $S_p$ be the leaf passing through $p$. By Lemma \ref{lmm_1}-(2), as $\phi$ is constant along $S_p$, so it is $\alpha$. Therefore, the immersed complex curve $S_p$ sits in $\Sigma^1_D\cap\{\phi=\phi(p)\}$ as well as in $\Sigma^1_X\cap\{\alpha=\alpha(p)\}$. But the latter, by the Main Theorem in \cite{mst} is either union of a finite number of irreducible compact curves or it is a real $3$-dimensional space foliated with complex leaves, which are then the only complex curves in it. Therefore $S_p=C_p$.\end{proof}

Putting these observations together with the geometric structure of $X$ given by the Main Theorem in \cite{mst}, we obtain the result we claimed.

\begin{teorema}\label{teo_teo}Let $X$ be a weakly complete complex surface endowed with a real analytic plurisubharmonic exhaustion function $\alpha:X\to\R$; consider $D\subseteq X$ a domain with a $\Ci^\infty$ plurisubharmonic exhaustion function $\phi:D\to\R$. Then, one of the following is true:
\begin{enumerate}
\item  $D$ is a modification of a Stein space
\item both $D$ and $X$ are union of complex curves
\item both $D$ and $X$ are Grauert type surfaces.
\end{enumerate}\end{teorema}
\begin{proof} We separate two cases.

\medskip

\noindent{\bf First case}

If we find a point $p\in\Sigma^1_D$ and a minimal function $\psi$ for $D$ such that the level set
$$Z_{\psi}(p)=\{q\in D\ :\ \psi(q)=\psi(p)\}$$
is regular for $\psi$, then, by Proposition \ref{prp_foglia}, $C_p\subseteq\Sigma^1_D\subseteq D$.

Obviously, also the closure of $C_p$ will be contained in $D$; in the case of Grauert type surfaces, this implies that $D$ contains the connected component of the level set of $\alpha$ which contains $p$.

\medskip

Now, if $X$ is fibered in compact curves, we have a (holomorphic) proper map with connected fibers $F:X\to T$, where $T$ is an open connected Riemann surface.
If $X$ is a surface of Grauert type with $M=\min\alpha$ of real dimension $2$, then there is a (pluriharmonic) proper map \cite{mst}*{Main Theorem} with connected fibers \cite{mst2}*{Theorem 3.1} $F:X\setminus M\to T$, where $T$ is the real line.
If $M$ has real dimension $3$, up to passing to a double cover $X^*$, we find a (pluriharmonic) proper map with connected fibers $F:X^*\to T$, where $T$ is an open connected subset of the real line.

\medskip

In all these three scenarios, we have a proper map with connected fibers $F:X^*\to T$, where $X^*$ is either $X$, $X\setminus M$ or a double cover of $X$ and $T$ is either an open connected Riemann surface or an open connected subset of the real line.

We define
$$\Omega=\{x\in T\ :\ F^{-1}(x)\subseteq D\}$$
and, by the first part of this proof, we know that $\Omega\neq\emptyset$; we also set $V=F(D)$. We know that $V$ is open and connected.

Consider a sequence $\{x_n\}\subset \Omega$, converging to a point $x\in V$. If the fibers of $F$ are irreducible compact complex curves, then obviously any plurisubharmonic function is constant on them. If the fibers of $F$ are smooth Levi-flat hypersufaces with dense complex leaves, then again any plurisubharmonic functio is constant on them (see the end of the proof of \cite[Corollary 3.8]{mst}); $\alpha$ being real analytic and proper, its critical values cannot accumulate, so, up to passing to a subsequence, we can suppose that $F^{-1}(x_n)$ is always contained in a smooth Levi-flat hypersurface. So, in both cases, as $F^{-1}(x_n)$ is contained in $D$, then $\phi$ is constant on it, for every $n$.

For every $p\in F^{-1}(x)\cap D$ (which is nonempty), there is a sequence of points $p_n$ such that $p_n\in F^{-1}(x_n)$, $p_n\to p$. This implies that $\phi$ is constant on $F^{-1}(x)\cap D$, but then $F^{-1}(x)\subset D$, because $\phi$ is an exhaustion function and the fibers of $F$ are compact and connected. Hence $\Omega$ is closed in $V$.

On the other hand, if $x\in \Omega$, then $\phi$ is constant on $F^{-1}(x)$, which is then a compact subset of $D$. By an easy topological argument, there exists a neighbourhood $U$ of $x\in T$ such that $F^{-1}(U)\subseteq D$, i.e. $U\subseteq\Omega$, i.e. $\Omega$ is open in $V$.

Therefore, $\Omega$ is both open and closed in $V$, which is connected; so $\Omega=V$, i.e. $D=F^{-1}(V)$.

\medskip

\noindent{\bf Second case}

Suppose that for every point $p\in \Sigma^1_D$, the level set $Z_\phi(p)$ is singular. By Sard's theorem, there exists a sequence $\{c_n\}_{n\in\mathbb{N}}\subseteq \R$ of regular values of $\phi$ such that $c_n\to+\infty$, hence there exists a sequence of regular level sets $Z_\phi(p_n)$ with $p_n\to bD$.

By hypothesis, none of these level sets intersects $\Sigma^1_D$, so, $\phi$ being minimal, all these levels are strongly pseudoconvex. By a standard argument, this implies that $D$ is a modification of a Stein space.
\end{proof}
	
\begin{rem}All the contents of this section, apart from the second case in the previous theorem, hold when $\phi$ is supposed to be only $\Ci^2$; however, in order to apply Sard's theorem, we need $\phi$ to be at least $\Ci^4$.\end{rem}

\section{Applications}

One special case of weakly complete surfaces with a real analytic plurisubharmonic exhaustion function are Hopf surfaces with one compact curve removed (see \cite{mst}*{Section 2} for a worked example). Applying Theorem \ref{teo_teo} to this particular case, we obtain a different proof for the Main Theorem in \cite{mie} and  Theroem 1.1. in \cite{LY}, under the hypothesis of the existence of a smooth plurisubharmonic exhaustion.

\medskip

More generally, suppose $X$ is the total space of a topologically trivial line bundle over a compact Riemann surface $M$; if the class in $H^1(M, \mathcal{S}^1)$ associated to the bundle is unipotent, then $X$ is fibered in compact curves, otherwise it is a Grauert type surface (see \cite{mst}*{Section 2} and \cite{mst2}*{Section 4} for related discussions and examples).

In both cases, we have a minimal plurisubharmonic whose only critical level is the minimum level, coinciding with $M$; so, any open set containing $M$ will also contain a regular level set of the exhaustion function.

Combining this observations with Theorem \ref{teo_teo}, we get the following

\begin{corol}Suppose $X$ is the total space of a topologically trivial line bundle $\xi$ over a compact Riemann surface $M$ and let $D\subset X$ be a weakly complete domain with a smooth plurisubharmonic exhaustion. Then one of the following happens:
\begin{enumerate}
\item $D$ is Stein
\item both $D$ and $X$ are fibered in complex curves
\item both $D$ and $X$ are Grauert type surfaces.
\end{enumerate}
\end{corol}
\begin{proof} From Theorem \ref{teo_teo} we know that, if cases (2) and (3) don't apply, then $D$ is a modification of a Stein space. However, if $D$ were not Stein, then it would contain some compact curves. If $X$ is a surface of Grauert type, the only compact curve in it is $M$, which doesn't have any Stein neighbourhood; on the other hand, if $X$ is fibered in complex curves, as soon as $D$ contains one of them, it is a union of curves, repeating the argument in the end of the first case in the proof of Theorem \ref{teo_teo}.

Therefore $D$ is Stein.
\end{proof}

\begin{rem}The proof we have just given applies whenever $X$ is a weakly complete surface which admits a real analytic plurisubharmonic exhaustion function whose only critical level set is the minimum level set.\end{rem}


\begin{bibdiv}
\begin{biblist}
%

%
%
%
%

%
\bib{LY}{article}{
author = {Levenberg, Norman},
author={Yamaguchi, Hiroshi},
doi = {10.2969/jmsj/06710231},
journal = {J. Math. Soc. Japan},
number = {1},
pages = {231--273},
publisher = {Mathematical Society of Japan},
title = {Pseudoconvex domains in the Hopf surface},
volume = {67},
year = {2015}
}


%
%

\bib{mie}{article}{
Title = {Pseudoconvex non-{S}tein domains in primary {H}opf surfaces},
Journal = {Izv. Ross. Akad. Nauk Ser. Mat.},
Author = {Miebach, C.},
Number = {5},
Volume = {78},
Year = {2014},
Pages = {191--200},
doi={10.1070/IM2014v078n05ABEH002717},
}

\bib{crass}{article}{
author={Mongodi, S.},
   author={Slodkowski, Z.},
   author={Tomassini, G.},
title = {On weakly complete surfaces},
journal = {Comptes Rendus Mathematique},
volume = {353},
number = {11},
pages = {969 -- 972},
year = {2015},
note = {},
issn = {1631-073X},
doi = {10.1016/j.crma.2015.08.009},

}


   \bib{mst}{article}{
   author={Mongodi, S.},
   author={Slodkowski, Z.},
   author={Tomassini, G.},
   title={Weakly complete surfaces},
   journal={Indiana U. Math. J.},
   date={2016},
   note={to appear}
   }
	
	\bib{mst2}{article}{
	   author={Mongodi, S.},
   author={Slodkowski, Z.},
   author={Tomassini, G.},
   title={Some properties of Grauert type surfaces},
	date={2017},
	journal={International Journal of Mathematics},
	note={to appear}}


%
%
%
%
\bib{ST}{article}{
    author={Slodkowski, Zibgniew},
   author={Tomassini, Giuseppe},
   title={Minimal kernels of weakly complete spaces},
   journal={J. Funct. Anal.},
   volume={210},
   date={2004},
   number={1},
   pages={125--147},
   issn={0022-1236},
   doi={10.1016/S0022-1236(03)00182-4},
}
%
%



%
  \end{biblist}
\end{bibdiv}
\end{document}